\documentclass[11pt]{article}
\usepackage{amsmath, amsthm, amssymb, setspace, mathrsfs, enumitem}
\newcommand{\partentry}[1]{\addtocontents{toc}
						{\small\bfseries#1\hfill\thepage\par}}
\makeatletter
\def\@part[#1]#2{%
    \ifnum \c@secnumdepth >\m@ne
      \refstepcounter{part}
      \partentry{\protect\makebox[2em][l]{\thepart}#1}
	\else
      \partentry{#1}
    \fi
    {\parindent \z@ \raggedright
     \interlinepenalty \@M
     \normalfont\Large\bfseries\thepart\hspace{1em}#2%
     \markboth{}{}\par}%
    \nobreak
    \vskip 3ex
    \@afterheading}
\def\@spart#1{%
    {\parindent \z@ \raggedright
     \interlinepenalty \@M
     \normalfont\Large\bfseries #1\par} 
     \nobreak
     \vskip 3ex
     \@afterheading}
\renewcommand\section{\@startsection{section}{1}{\z@}%
 						{-3.5ex \@plus -1ex \@minus -.2ex}
						{2ex \@plus.2ex}
						{\large\bfseries}}
\renewcommand\subsection{\@ifstar
						{\setcounter{subsection}{\value{equation}}
					\@startsection{subsection}{2}{\z@}
                          {1.75ex \@plus.5ex \@minus.2ex}%
                           {-.4em}		
					\textit*}
					{\setcounter{subsection}{\value{equation}}
						\stepcounter{equation}
					\@startsection{subsection}{2}{\z@}
                          {1.75ex \@plus.5ex \@minus.2ex}%
                           {-.4em}		
					\textit}}
\def\@seccntformat#1{\@ifundefined{#1@cntformat}%
	{\csname the#1\endcsname\quad} 
	{\csname #1@cntformat\endcsname}} 
\def\section@cntformat{\thesection.~} 
\def\subsection@cntformat{(\thesubsection)\ }
\renewcommand*\l@section{\mdseries\small\@dottedtocline{1}{1.5em}{2em}}
\makeatother

\numberwithin{equation}{section}
\theoremstyle{plain}
\newtheorem{maintheorem}{Theorem}
\newtheorem{conjecture}{Conjecture}
\swapnumbers
\newtheorem{theorem}[equation]{Theorem}

\newtheorem{lemma}[equation]{Lemma}

\theoremstyle{definition}
\newtheorem{definition}[equation]{Definition}
\theoremstyle{remark}
\newtheorem{remark}[equation]{Remark}

\newcommand{\cA}{\mathscr{A}}

\newcommand{\cC}{\mathscr{C}}

\newcommand{\cF}{\mathscr{F}}

\newcommand{\cL}{\mathscr{L}}

\newcommand{\frg}{\mathfrak{g}}

\newcommand{\frt}{\mathfrak{t}}

\newcommand{\bC}{\mathbb{C}}
\newcommand{\bF}{\mathbb{F}}

\newcommand{\bQ}{\mathbb{Q}}

\newcommand{\bZ}{\mathbb{Z}}

\begin{document}
\title{\textbf{Quantization commutes with reduction again: \\
the quantum GIT conjecture I}}
\author{Daniel Pomerleano\footnote{Partially supported by NSF grant DMS-2306204}, Constantin Teleman\footnote{Partially 
supported by the Simons Collaboration of Global Categorical Symmetries}}
\date{\today}
\maketitle
\begin{quote}
\abstract{
\noindent 
For a compact monotone symplectic manifold $X$ with Hamiltonian action of a compact Lie group $G$ and smooth symplectic 
reduction, we relate its gauged $2$-dimensional $A$-model to the $A$-model of $X/\!/G$. This (long conjectured) result is 
parallel to the ($B$-model!) \emph{quantization commutes with reduction} theorem of Guillemin and Sternberg in quantum 
mechanics. Here, we spell out some of the precise statements, and 
outline the proof of equality for the spaces of states (quantum cohomology). We also indicate the way to some related results 
in the non-monotone case. Additional Floer theory details will be included in a follow-up paper. 

}
\end{quote}

\section*{Main Results}

\subsection*{Basic case.} In Theorem~\ref{main}, we equate the \emph{gauged} quantum cohomology of a compact 
$G$-monotone\footnote{See below for our definition of $G$-monotone, stronger than monotone when $G$ 
is non-abelian.} 
symplectic manifold $X$ under a Hamiltonian action of a connected compact group $G$ with that of its symplectic reduction at the 
anticanonical linearization: 
\begin{equation}\label{loopgp}
QH^*_{LG}(X)\cong QH^*(X/\!/G). 
\end{equation}
This assumption was present, at least implicitly, in the 
physics-inspired literature on the $A$-model QFT of the 
1990s, in strong analogy with the quantum mechanics situation, recalled in \S\ref{back}. It was the background of 
early quantum cohomolgy calculations, such as Batyrev's formula \cite{bat}.

The subscript $LG$ indicates equivariance under the \emph{free loop group} of $G$; its action on the Floer complex will be 
described in \S\ref{topology}. The gauged side thus incorporates twisted sectors, as in the case of finite groups. The 
statement is subject to the known constraint that the quotient should be free: even the orbifold 
case requires multiplicative adjustment, 
although an additive isomorphism holds. 

\subsection*{Formality.} With complex (or rational) coefficients, a stronger \emph{formality} result describes $QH^*(X/\!/G)$ as a quotient 
of $QH^*_G(X)$:
\begin{equation}\label{formal}
QH^*(X/\!/G) \cong QH^*_G(X)\otimes _{H_*^G(\Omega G)} H^*(BG).
\end{equation}
The tensor product is strict, with no higher Tor groups. Going forward, we will use complex coefficients in cohomology 
without further comment.

Recall that $\mathrm{Spec}\, H_*^G(\Omega G)$, with its projection to $\mathrm{Spec} H^*(BG)$, is the total space of the 
\emph{Toda integrable system} (for the Langlands dual group $G^\vee$), a completely integrable algebraic symplectic manifold \cite{bfm}. Formula~\eqref{formal} restricts $QH^*_G(X)$ to the 
unit section $\mathrm{Spec}\,H^*(BG)$ of the Toda space. This stronger result is a long intended application of the theory of 
gauged categories and 2D TQFTs \cite{telicm}; additional details were described over time in numerous lectures, but the lack of as 
 single comprehensive account requires a slew of self-citations  \cite{telaus, telc}, for which the second author apologizes.  
The remarkable success of algebro-geometric techniques in Gromov-Witten theory was partly responsible for the delay in 
execution: it took time to accept this as a genuine Floer theory result. We do not know an algebro-geometric proof. 

\subsection*{Prior work.}
Our argument only addresses the additive isomorphism, and we rely on earlier work of Wehrheim-Woodward \cite{ww}, Fukaya 
\cite{f, f2} and, most recently, Xiao \cite{xiao} to transfer the multiplicative structure: this exploits the equivariant Lagrangian correspondence defined by the 
zero-fiber of the moment map $\mu$. Monotonicity and freedom of the $G$-action are essential for this part of the argument. Removing those assumptions leads to deformations in the multiplicative structures. 

\subsection*{Quantum Martin formula.} When the symplectic quotients of $X$ under $G$ and its maximal torus~$T$ are both 
smooth, their (classical) cohomologies are related by an explicit formula due to Shaun Martin~\cite{mar}. 
For projective manifolds, a quantum version of that formula was established by Gonzalez and Woodward \cite[Thm.1.6]
{gw}. For us, that relation is, instead, a formal 
consequence of the relation between the Toda spaces for $G$ and $T$---the \emph{string topology domain wall} of 
\cite[\S5]{telicm}---which control the respective versions of Formula~\eqref{formal}. The annihilator ideal in Martin's 
description is the portion of $LT$-equivariant quantum cohomology which is \emph{not} fully $LG$-equivariant: it is intercepted 
by the exceptional fiber of the Toda integrable system, but not by the unit section $\mathrm{Spec}\,H^*(BG)$.

\subsection*{Trapped cohomology.}
Our argument relies on \emph{Floer continuation} under Hamiltonians $K\|\mu\|^2$, and part of it extends to the case 
where $X$ is monotone, but the $G$-linearization is not anticanonical: namely, $LG$-descent 
from the stable Morse stratum in $X$ to $X/\!/G$ always holds.  Then, $QH^*(X/\!/G)$ is additively expressed as 
a quotient of $QH^*_{LG}(X)$, with kernel the \emph{trapped cohomology}: this is defined from a Floer subcomplex of classes 
which stop at critical sets under Floer continuation. This subcomplex is an ideal; however, just as in \cite{f, f2, xiao}, 
one sees a bulk deformation of the multiplicative structure of $QH^*_G(X)$, vis-\`a-vis the quotient of $QH^*_{LG}(X)$. 
Since $X/\!/G$ is generally not monotone, its $QH^*$ involves a Novikov parameter, which appears in the deformation class.

We do not know a description of the trapped ideal in terms of fixed-point localization in~$X$; 
this is surely related to the difficulty of an algebro-geometric approach. The trapped part can be characterized in terms of the 
symplectic action, and is nicely connected with recent work of Varolgunes \emph{et al.}~\cite{var1, var2} on symplectic 
cohomology with supports: for instance, in Theorem~\ref{supports}    we equate $QH^*(X/\!/G)$ additively with the $LG$-equivariant 
symplectic cohomology of $X$ with supports near $\mu^{-1}(0)$. 
That discussion will be developed in our follow-up paper \cite{pt}, so the statement for now functions as an announcement.  

\subsection*{The non-monotone case.} 
Assuming a genuinely equivariant Floer theory,\footnote{See the \emph{Floer techniques} paragraph below.} 
the last part of our story does not even require monotonicity of~$X$. There is always a short exact sequence 
\[
 0\xrightarrow{\ } tHF^*_{LG}(X) \xrightarrow{\ } QH^*_{LG}(X) 
	\xrightarrow{\ } QH^*(X/\!/G) \xrightarrow{\ } 0
\]
describing the quotient \emph{aditively} in terms of $LG$-equivariant and trapped Floer complexes of~$X$. As before, the 
multiplication on the quotient is deformed. 

One additional difficulty here arises when using a \emph{formally complete} Novikov base ring  
for the Floer complex. This conceals the trapped part $tHF^*_{LG}(X)$: the torus of Seidel shift operators degenerates in the 
classical limit, their spectrum becoming adically disconnected by the Novikov variable; the trapped part is lost to  
infinitely large, or small, eigenvalues and does not appear  perturbatively in the Novikov parameter. This 
is may be why this geometric trapping phenomenon was not flagged in prior work, even though the mismatch in ranks of quotients, 
in toric mirror symmetry calculations, reveals it clearly  (see also Remark~\ref{conv}). 

\subsection*{Non-abelian $G$-monotonicity.} We should now confess to a limitation of our present proof in the case of 
\emph{non-abelian} $G$. The simple argument we outline here, based on a \emph{monotone index} estimate, fails for 
non-abelian groups when $X$ has fixed-points with moment image  too close to zero. For instance, when 
$G=\mathrm{SU}(2)$, we must avoid fixed points with weights $1$ or $2$. Generally, we require \emph{no wall-crossings}  
between the reductions of $X$ and those of $X\times F$, for the  flag varieties $F$ of $G$.    
Equivalently, all dominant co-weights $\xi$ of $G$ must satisfy, for all $\xi$-fixed-points $x\in X$, 
\begin{equation}\label{supermonotone}
\langle \mu(x) |\xi\rangle > 2\langle \rho \,|\, \xi\rangle.
\end{equation} 
If \eqref{supermonotone} fails, some $G$-equivariant cohomology \emph{may well be trapped} at $x$. 
Such is the case for flag varieties of $G$, or (say) for the $\mathrm{SU}(2)$ action on the 
Hirzebruch surface $\bF_1$, when \emph{all} cohomology is trapped. Morally, the positivity of a small co-adjoint 
orbit breaks the $G$-transversal positivity of~$X$. Remarkably, the statements 
\eqref{loopgp}, \eqref{formal} still seem to hold: a more sophisticated argument, relying on the Rietsch mirrors of flag varieties, indicates the absence of \emph{$LG$-equivariant} 
trapped cohomology. We shall postpone that discussion to as subsequent paper, as it requires reconciling mirrors of flag 
varities, as in \cite[\S6]{telicm}, with our analysis of Floer continuation.

\subsection*{No Floer differentials?} One undertone of the results is the possible \emph{formality} of 
the $G$-equivariant Floer complex of $X$ after deformation by Hamiltonians $K\|\mu\|^2$. Namely, after a gauge transformation 
of the complex,\footnote{This is expected to be the conjugation by $\exp(W/\hbar)$, where the superpotential $W$ is the fundamental solution to the $D$-module 
defined by $QH^*_G(X)$, and $\hbar$ is the equivariant parameter for loop rotation.} we expect the only 
differentials on the spaces of orbits to be the topological ones that execute the passage from plain to equivariant 
cohomology. It seems difficult to formulate a robust claim, though; for instance, $K\|\mu\|^2$ need not satisfy 
the Morse-Bott property required to \emph{define} the Floer complex. For torus groups, 
a small shift of $\mu$ addresses this, but other  perturbations are needed  in the non-abelian case, possibly leading to additional critical sets and  differentials.

\subsection*{Floer techniques.} If equivariant Floer theory works as expected, our statements carry straightforward 
generalizations, such as compatibility with the $E_2, E_3$ structures of TQFT, and the arguments in our paper settle 
that. However, working outside the monotone case requires the technique of virtual fundamental classes, for which 
the equivariant case is not yet fully developed. Without these tools, the case of monotone $X$ admits the work-around of 
family Floer methods over $BG$. (See \S\ref{topology} for the starting definition, but details will appear in \cite{pt}.) 
Thus, all claims in this paper concerning non-monotone $X$ are \emph{conditional on a good equivariant Floer theory}: see 
e.g.~\cite[Assumption 1]{xiao}.

\subsection*{History.}
The \emph{conjectural} form of Theorem~\ref{main} has long featured in lectures by the second author, as a key application of 
the character theory of categorical group representations \cite{telicm}. It was first 
announced as a theorem, with the outline of our current proof, in \cite[iv]{telaus}; a plan for it had been proposed in \cite[iii] {telaus}, but key Floer ingredients were missing.  Some elements of our construction have 
since appeared elsewhere, but we 
believe ours to be the first unified treatment.  

\subsection*{Acknowledgements.}
Several colleagues---among whom we mention Denis Auroux, Kenji Fukaya, Hiroshi Iritani, Paul Seidel, Chris Woodward---contributed 
comments and suggestions; most importantly, on the role of monotonicity and freeness of the 
action, the role of bulk deformations, possible relations to the quantum Kirwan map and the quantum blow-up formulas. We gratefully 
acknowledge their help.

\section{Background and Statements} \label{back}
The famous `quantization commutes with reduction' theorem of Guillemin and Sternberg \cite{gs2} expresses the commutation of 
the \emph{geometric quantization} \cite{so, k} of a classical phase space with \emph{reduction} (quotienting out) under a compact 
group of symmetries. Precisely stated, let $(X, \omega)$ be a compact symplectic manifold and $\cL\to X$ a \emph{prequantum 
line bundle} (a complex line bundle with hermitian connection with curvature $2\pi i\omega$). Assume that $X$ has been equipped with 
a positive, integrable complex polarization of $\omega$. Then, $\cL$ becomes an ample holomorphic line bundle, 
and~$X$ a projective K\"ahler manifold. The holomorphic sections assemble to the Hilbert space $\Gamma(X;\cL)$ of the 
geometric quantization. 

Let now a compact, connected Lie group $G$ act isometrically on $X$ and $\cL$; the lift defines Hamiltonians for the 
infinitesimal action, captured by a \emph{moment map} $\mu:X\to \frg^*$. Work of Hesselink, Kempf-Ness and Kirwan \cite{h, kn, k} 
equate Mumford's \emph{GIT quotient} of $X$ by the complexification of $G$, defined algebraically  \cite{mum} as 
$\mathrm{Proj} \bigoplus\nolimits_{n\ge 0} \Gamma(X;\cL^n)$, with the symplectically constructed \emph
{Marsden-Weinstein-Meyer reduction} $X/\!/G:= \mu^{-1}(0)/G$  \cite{mw, mey}. The former is manifestly algebraic, the latter symplectic, leading, when $G$ acts freely on $\mu^{-1}(0)$, to a K\"ahler structure and a descended holomorphic line bundle $\cL/\!/G$. 

\subsection{Dimension $1$ results: Quantum mechanics.} 
A restriction map 
\[
\Gamma(X;\cL)^G \to \Gamma(X/\!/G;\cL/\!/G)
\]
is mediated by $\mu^{-1}(0)$; Guillemin and Sternberg show this to be an isomorphisms. The result has vast generalizations, 
to almost complex structures \cite{mein} or to higher cohomology \cite{telq}. Anachronistically, we could call this a \emph{B-model} statement, as it pertains to holomorphic data, despite the intermediate 
appearance of symplectic forms.  

The companion \emph{A-model} result is due to Kirwan \cite{k}. It does not require integrability of the complex structure, but 
uses an invariant metric on $\frg$. Viewing $\|\mu\|^2$ as a Morse function leads to a $G$-stratification of 
$X$, with lowest stratum retracting equivariantly to $\mu^{-1}(0)$. Kirwan shows the stratification to be \emph{equivariantly perfect} for rational cohomology: it defines a filtration on $H^*_G(X;\bQ)$ with quotients the $G$-equivariant cohomologies with supports on the strata. The top quotient $H^*_G\left(\mu^{-1}(0)\right)$ agrees with $H^*(X/\!/G)$ when the action is locally free. 

\subsection{Difference betwen the statements.} The second ($A$-model) conclusion may seem less sharp. The underlying reason 
lies in the nature of the two representation theories. The $B$-model statement pertains to linear actions of $G$ on 
(complexes of) vector spaces in the usual sense (`$B$-model actions'). 
Kirwan's statement pertains to \emph{topological representations}, most easily defined as (derived) local systems over the 
classifying space $BG$. Here, the sections of the local system represents the equivariant cochains on $X$. Now, whereas $B$-model 
representation theory is semi-simple, the $A$-model version is entirely `derived' when $G$ is connected: the cohomology of 
a space is invariant under any connected group of automorphisms, and the action is concealed in differentials and extension 
classes. Hoping for the (derived) invariants to project, out of $H^*(X)$, the desired part $H^*(X/\!/G)$, would be far too 
optimistic; the filtration conclusion is optimal. 

\subsection{Dimension $2$ and TQFT.}
Surprisingly, the stories switch for $2$-dimensional quantization. Here, the $B$-model representation theory is derived, while the 
$A$-model gives an exact answer (admittedly, subject to strong positivity assumptions). The reason is traced to the rigid (complex holomorphic) \emph{character theory} for topological $G$-actions on linear categories \cite{telicm}. 

Let us first recall the $B$-model story, with $X,\cL$ and $G$ be as before. Lurie's Cobordism Hypothesis \cite{l} ensures that the 
derived category $D^b\cC oh(X)$ generates a fully local, $2$-dimensional topological quantum field theory (TQFT) for framed 
manifolds. In this sense, $D^b\cC oh(X)$ replaces the space of sections $\Gamma(X;\cL)$ of quantum mechanics. Consider the 
$\|\mu\|^2$-Morse stratification, which turns out to be algebraic (and agrees with the instability stratification \cite{kn, k}). 
The analogue of Kirwan's filtration, proved by Halpern-Leistner \cite{hl} and independently (in a slightly 
different formulation) by Ballard, Favero and Katzarkov \cite{bfk} is a \emph{semi-orthogonal decomposition} of the 
$G$-equivariant version of $D^b\cC oh(X)$: one component is $D^b\cC oh(X/\!/G)$, with the other components associated to 
the unstable strata. (See \emph{loc.~cit.} for precise statements.)

Turning to the $A$-model, we abandon the integrability of the polarization, landing in the world of $J$-holomorphic 
structures; but we place the strong restriction that $\cL = \det(TX)$ (in the almost complex structure), with its natural 
$G$-action. Equivalently, $\omega$ is in the class of $c_1(TX)$ (monotonicity) and the moment map defines an equivariant 
refinement of that relation: $[\omega -\mu] = c_1\in H^2_G(X)$. We can then define the Fukaya category $\cF(X)$ over the 
Laurent polynomial ring in a single variable $q$ of degree $2$; in parallel, the quantum multiplication corrects the classical 
one on cohomology by polynomial terms in $q$. The category 
$\cF(X)$ is believed to generate a $2$D TQFT for oriented 
surfaces, reproducing Gromov-Witten theory in dimensions $1$ 
and $2$. This leads to the  

\begin{conjecture}\label{fconj}
The Lagrangian correspondence $\mu^{-1}(0) \subset X\times X/\!/G$ induces an equivalence of categories 
$\cF(X)^G \equiv \cF(X/\!/G)$. 
\end{conjecture}

\noindent
The methods we use for our main theorem here are expected to affirm this conjecture, but we do not address it here, 
for two reasons:
\begin{enumerate}\itemsep0ex 
\item While it is easy to define the topological $G$-action on the Fukaya category $\cF(X)$ by Floer methods, 
the (homotopy) fixed-point category $\cF(X)^G$ requires some care: an \emph{uncompletion} is needed,\footnote{Attempts to define it 
using $G$-invariant Lagrangians produce incorrect answers.}  as indicated in \cite[v]{telicm,telaus}.
Details of the construction will appear in a paper by the second author. 
\item The category $\cF(X)$ is not known to be smooth and proper in general, nor is its Hochschild cohomology 
known to agree with the quantum cohomology $QH^*(X)$; so not many Gromov-Witten consequences could 
be immediately extracted\footnote{For GW theory, one must supplement the category with a splitting of the 
Hochschild-to-cyclic sequence, a mysterious ingredient for now. See, however, recent work of Iritani \cite{i} for circle actions, with remarkable application to the blow-up conjecture.} from Conjecture~\ref{fconj}. Rather the opposite: 
quantum cohomology controls deformations of $\cF(X)$ via the 
closed-to-open map, so consequences for $\cF(X)^G$ may be derived from $QH^*_G(X)$.
\end{enumerate}

\subsection{New results.}
We address instead quantum cohomology directly, proving the following \emph{twisted sector} description of the quantum cohomology 
of $X/\!/G$. This is analogous to Batyrev's recipe for the quantum cohomology of toric Fano varieties, when $X$ is a linear 
representation of a torus $G$. (Our argument may in fact be adapted to that situation.) Cohomology has complex coefficients.

\begin{maintheorem} \label{main}
Let $X$ be compact, symplectic, $G$-monotone, with Hamiltonian action of the compact group $G$, linearized by $\det(TX)$. 
\begin{enumerate}\itemsep0ex
\item The equivariant quantum cohomology $QH^*_G(X)$ is an algebra over the equivariant homology $H_*^G(\Omega G)$ (with its Pontryagin product).
\item The Lagrangian correspondence $\mu^{-1}(0)$ defines a $QH^*_G(X)-QH^*(X/\!/G)$ bimodule \cite{ww, f}.
\item When $G$ acts freely on $\mu^{-1}(0)$, this bimodule induces an algebra isomorphism 
\[
QH^*(X/\!/G) \cong QH^*_G(X) \otimes_{H_*^G(\Omega G)} H^*(BG);
\]
the tensor product is strict.
\item When the $G$-action is only locally free, we obtain an \emph{additive} isomorphism in (iii).
\end{enumerate}
\end{maintheorem}

\begin{remark} {\ }
\begin{enumerate}\itemsep0ex
\item We need the \emph{polynomial}, not power series version of the cohomology of $BG$; cf.~\S\ref{resfp} below.
\item Calculation of the tensor product can be reduced to the maximal torus and Weyl group of~$G$; this reproduces  the \emph{quantum Martin} formulas 
of \cite[Theorem~1.6]{gw}.
\item Under the assumption that $QH^*$ is the Hochschild cohomology of (a good version of) the Fukaya category, 
Part~(i) goes back to \cite{telicm}. Without that assumption, the \emph{twisted sector} construction was described 
in the IAS lecture \cite[iii]{telaus}. Its analogue for the \emph{symplectic cohomology} of a linear $G$-representation $X$ was predicted 
in \cite{telc}, and recently proved in detail in \cite{gmp}. That example is relevant to the Batyrev formula and its non-abelian 
generalizations.  
\item The Lagrangian correspondence in Part~(ii) appears in the work of Wehrheim and Woodward \cite{ww} in the monotone case, and 
was further studied by Fukaya \cite{f} more generally. Woodward \cite{w} also defined an $A_\infty$ \emph{quantum Kirwan map} 
$QH^*_G \to H^*(X/\!/G)$. We expect the latter to factor through our isomorphism, but the difference in methods makes this 
not so obvious. 
\item Assuming good properties of equivariant Floer homology, 
the correspondence inducing the isomorphism provides a chain-level lift, with the appropriate algebra and commutation ($E_2$) 
structures. All we prove in this paper is the additive statement, and we rely on the correspondence for the rest of the structure.
\end{enumerate}
\end{remark}

\begin{remark}[Gauge theory and boundary conditions]
The variety $\mathrm{Spec}\, H_*^G(\Omega G)$ is the (fiberwise group completion 
of the) classical Toda integrable system. 
The tensor product in Theorem~\ref{main}.iii is the restriction of the algebra $QH^*(X/\!/G)$ to the unit section. 
The holomorphic Lagrangian geometry of the Toda group scheme controls boundary conditions for $3$-dimensional $A$-model 
gauge theory \cite{telaus, telicm}. The unit section corresponds to the unit (Neumann) boundary condition, and 
its intersection with $QH^*_C(X)$ computes the space of states for the gauge-invariant part of the $A$-model TQFT; this is expected to be defined by (a good version of) the $G$-invariant Fukaya catgory $\cF(X)^G$, when the latter has good properties  (cf.~Conjecture~\ref{fconj} and the remarks that follow).  
\end{remark}

Monotonicity of $X$ and the anti-canonical linearization are essential for Theorem~\ref{main}. Nonetheless, 
our conclusions survive in part even without those assumptions. This is clearest when  $X$ is monotone, but the $G$-linearization 
is shifted, and we still work over the ground ring $\bQ[q^{\pm}]$. When $X$ itself is not monotone, we must replace 
that by a Novikov ground ring from the outset (cf.~Remark~\ref{conv} below). Most importantly, results for non-monotone $X$ are \emph{conditional} on a good behavior of 
equivariant Floer theory; see for instance  \cite[Assumption 1]{xiao}. In both cases, a \emph{bulk deformation 
class} appears in the multiplicative structures  \cite{f, f2, xiao}. 

\begin{maintheorem}\label{extra}
Without the monotone assumption in Theorem~\ref{main}, we have an exact sequence
\[
 0 \xrightarrow{\ }  tHF^*_{LG} \xrightarrow{\ } QH^*_{LG}(X) \xrightarrow{\ } A^*(X/\!/G) \xrightarrow{\ } 0,
\]
where $tHF^*_{LG}(X)$ is the ideal of gauged $LG$-equivariant \emph{trapped Floer cohomology} (cf.~\S\ref{trapcoh} below). 
After possibly passing to a suitable Novikov ring, $A^*$ admits a bulk deformation to  $QH^*(X/\!/G)$.
\end{maintheorem}

We do not know explicit multiplicative splittings of this sequence, although specific ones are found in important 
examples, such as in Iritani's work on the blow-up formula \cite{i}. The deformation class from $A^*$ to $QH^*(X/\!/G)$, 
which requires passing a Novikov ground ring if we started out over $\bQ[q^{\pm}]$, also seems difficult to compute from geometric data. 
However, without deformation or change in ground ring, we have:

\begin{maintheorem}\label{supports}
The quotient algebra $A^*(X/\!/G)$ in Theorem~\ref{extra} can be identified with the $LG$-equivariant symplectic cohomology of $X$ 
with supports on $\mu^{-1}(0)$.
\end{maintheorem}
\noindent
Symplectic cohomology with support was introduced by Varolgunes' \cite{var1}. When $X$ is not monotone, defining its the correct $LG$-equivariant version for this theorem also 
presupposes a working equivariant Floer theory. In any case, further discussion of Theorem~\ref{supports} will await \cite{pt}. 

\subsection{Trapped cohomology.}\label{trapcoh}
The presence of a kernel $tHF^*_{LG}$ has long been known in toric, non-Fano mirror symmetry formulas; here, we provide the 
geometric explanation. Part of the complex gets trapped at critical points under \emph{Floer continuation} by the Hamiltonian 
$K\|\mu\|^2$, $K\to \infty$. This entrapment is unsurprising, being intrinsic to Morse theory; more surprising is the 
`clearing' of certain critical points by the Floer complex (see the ``movie'' of \S\ref{movie}).  

The underlying reason is remarkably similar to the `geometric quantization commutes with reduction' argument \cite{telq}. In 
that case, $G$-invariant Dolbeault cohomology with supports on the unstable strata was ruled out by the action of certain 
\emph{Hilbert-Mumford destabilizing circles}. In our symplectic situation, the shift operators of those circles 
play a decisive role: in the $G$-monotone case, all $G$-equivariant Floer cohomology is extracted out of the 
critical points by Floer continuation under the same subgroups, and the trapped part can 
be described by shift operators: the downward shifts that should extract the cohomology into the stable locus keep it in place 
instead. 

Thus, $tHF^*_{LG}$ carries more structure, being filtered by  unstable Morse strata. 
Trapping are those strata on which the Hilbert-Mumford destabilizing circle acts with non-positive  
weights on the fiber of $\det(TX)$. In work going forward, we plan to study this analysis of Floer 
continuation through critical points. However, the monotone case allows for a robust simplification of the 
argument, which we shall use in this paper. 

\begin{remark}[Formality and convergence]\label{conv}
The Novikov ring includes power series in a variable $v$, 
raised to the (negative) symplectic areas of curves. The most general setting of Floer theory uses formal series. If we can 
work instead with \emph{convergent series}, the trapped ideal is typically non-zero, the 
extreme case being that of empty quotients. However, formal completion may annull it: the spectrum of the shift operators 
becomes adically separated in $v$, according to the chamber structure of the moment map, which controls their 
classical degeneration. A choice of reduction chamber will discard the associated $tHF^*_{LG}(X)$. This reveals the limitations of  
working with a formal ground ring: the quotient $A^*(X/\!/G)$ ``jumps inexplicably'' when varying the moment reduction 
parameter, whereas the gauged theory $QH^*_{LG}(X)$ should vary analytically---as confirmed in numerous mirror symmetry 
calculations.  

We expect that an \emph{additive} $LG$-equivariant Floer complex of $X$ can always be defined over a convergent version 
of the Novikov ring, but arguments for that presuppose the use of virtual fundamental classes. An analogue for the multiplicative structures seems not known.
\end{remark}

\subsection{Contents of the paper.}
In the following sections, we spell out the ingredients and sketch the proof of Theorems~\ref{main} and \ref{extra}.i. The 
methods are Floer-theoretic, but our summary here will use, without comment, 
equivariant Floer complexes as if the $G$-action was strict everywhere and the Hamiltonians $K\|\mu\|^2$ were  
Morse-Bott on the Floer complex. The companion paper \cite{pt} will spell out the proper workaround, in terms of family Floer 
complexes and their action filtration. In the non-abelian case, the most important changes arise when 
abandoning $G$-monotonicity; that will be discussed in a separate paper.     

\section{Topological (loop) group actions}\label{topology}
We describe the topological origin of the algebraic structures in Theorem~\ref{main}. The notion of \emph{restricted homotopy 
fixed points} sets the stage for equivariant Floer calculus in the monotone case, avoiding strict equivariance. 
On the way, we spell out the argument for Theorem~\ref{main}.i.  

\subsection{Loop group action on the Floer complex.}\label{lgacts} 
The Floer complex of $X$ is a model for semi-infinite chains on the free loop space $LX$. With complex coefficients, we can vary 
the connection by classes transgressed from $H^2(X;\bC^\times)$, the parameter space for small quantum cohomology. The grading 
has an additive ambiguity of $2\min\langle c_1 | H_2(X)\rangle$, which can be resoved by working over the Laurent polynomial ring in a 
variable $q$ of degree $2$. Valuing the quantum parameter more broadly in $H^2(X;\bC^\times)$ incorporates a varying choice 
of coefficients: the flat complex line bundles over $LX$, which are $(E_2)$ multiplicative for the pair-of-pants product. 
And indeed: 

\begin{remark}[Orientation]
The Floer complex takes values naturally in the orientation bundle $o_{LX}$ of $LX$. As explained in \cite{at}, the latter is 
classified by the trangression to $H^1(LX;\bZ/2)$ of the Stiefel-Whitney class $w_2(X)$. The latter represents the $2$-torsion 
point $\exp(\pi i c_1)$ in the identity component of $H^2(X;\bC^\times)$. As a \emph{complex} line bundle, $o_{LX}$ is 
trivial, and its connection can conceal the twisting contribution to Floer differentials. When $w_2\neq 0$, odd Deck-transformations  are orientation-reversing, which can give rise to unexpected signs in the Floer complex.  
\end{remark}

An action of $G$ on $X$ should define a topological action of $LG$ on the Floer complex, leading to the definition of the $LG$-
equivariant quantum cohomology $QH_{LG}^*(X)$, and it does:

\begin{lemma} \label{action} Let $G$ be connected.
\begin{enumerate}\itemsep0ex
\item A topological action of $LG$ on a co-chain complex $F$ is equivalent, up to $G$-equivariant homotopy, to a (derived) local 
system $\widetilde{F}$ over $G$, with fiber quasi-isomorphic to $F$, together with equivariance data for $G$-conjugation.
\item This in turn is equivalent to a $G$-action on $F$, plus a $G$-equivariant action of $\Omega G$ on $F$.
\item Equivariance under loop rotations of $LG$ and a matching circle action on $F$ leads to circle-equivariance in (i) and (ii).  
\end{enumerate}
\end{lemma}
\begin{proof}
Part~(i) holds because the quotient stack $G/_{\mathrm{Ad}}\,G$ is a model for the classifying space of $LG$, as can be seen by the 
former's presentation as the stack of $G$-connections on the trivial bundle over the circle. Part~(ii) just spells out semi-
direct decomposision $LG\cong G\ltimes \Omega G$; it is also, of course, related to (i) via the monodromy representation with base 
point the unit in $G$.

Part (iii)  is now clear, except perhaps for spelling out the circle action in (ii): it comes from equating 
$\Omega G/_{\mathrm{Ad}} G = G\setminus LG/ G$. The circle 
actions are the natural ones. \end{proof}

The requisite $LG$ action on the Floer complex of $X$ comes from the local system over $G$ whose fiber at $g\in G$ is the Floer 
chain complex $CF^*(X;g)$ of the symplectomorphism defined by $g\in G$. This is equivariant for conjugation, with the fiber at 
$g$ carrying the action of the centralizer $Z(g)$. The local system structure---or, equivalently, the acion of $\Omega G$---comes form interpreting paths in $G$ as Hamiltonian isotopies. 

\begin{lemma}\label{derived}
The equivariant homology of $G$ with $\widetilde{F}$ coefficients is given by the (left derived) tensor product 
\[
C^{\,G}_*(F) \otimes^L _{C^{\,G}_*(\Omega G)} C_*^{\,G}
\]
\end{lemma}
\begin{proof}
We just identify homology over $G$ with co-invariants of monodromy, fiberwise over $BG$. 
\end{proof}

\begin{remark}
The Floer complex has a natural $E_2$-multiplicative structure form the \emph{pair of pants product}. 
This is compatible 
with the Pontryagin product on $G$, and replicated on $H_*^G\large(G;\widetilde{CF}{}^*\large)$ by considering the moduli of flat bundles on the pants with boundary monodromies $g, h, gh$. The tensor product in Lemma~\ref{fixedpoint} above involves 
two $E_2$ algebras over a central $E_3$-algebra, and the  algebraic and geometric $E_2$ structures match. Following an insight of 
Kontsevich and Costello, the circle action by loop rotation on the Floer complex of a compact symplectic manifold is trivialized 
by an extension of Gromov-Witten operations to the boundary of the genus-zero moduli space~$M_0^n$;  this also trivializes the bracket part of the $E_2$ structure. We do not 
know a reference for the equivariant version; it does follow by fibering $X$ over $BG$---in effect, from the exisence of $G$-equivariant Gromov-Witten theory.  
\end{remark} 

\subsection{Restricted homotopy fixed points.} \label{resfp}
The homotopy $G$-invariant part of a $G$-complex $F$, or of the $F$-valued chains in $G$ as in Lemma~\ref
{action}, is its derived space of sections over $BG$. It is a module over 
$C^*(BG)$, and its cohomology is complete at the maximal ideal. We will need the \emph{polynomial} version 
instead (sometimes, misleadingly, called the equivariant \emph{homology} $H^G_*$). 
The distinction is meaningful, because of the \emph{a priori} grading collapse in the Floer complex. 

A general $\mathbf{R}\Gamma(BG;F)$ need not have preferred uncompletions: but this will be the case if $F$ 
is cohomologically bounded below. Intuitively, this is obvious: only a finite-dimensional part of $BG$ 
contributes to $H^*(BG;F)$ in a fixed degree, and the graded direct sum is the prefered uncompletion. 
More precisely, we present such an $F$ as a colimit of the Postnikov fibers, subcomplexes $\tau^{\le n}F$, 
vanishing above degree $n$ and agreeing with $F$ below; these are defined  
up to quasi-isomorphism over $BG$, and are finite $BG$-extensions of local systems. 

\begin{definition} \label{fixedpoint}
If $F$ is cohomologically bounded below, the \emph{restricted homotopy invariant complex} is the colimit of the 
$\mathbf{R}\Gamma(BG; \tau^{\le n}F)$. For a colimit of bounded-below complexes, the homotopy invariant 
complex is the corresponding colimit.   
\end{definition}

\begin{remark}
The colimit presentation is needed as part of the data. In geometric situations, a preferred one is available: from 
a \emph{strict} $G$-action on a space $S$ (even infinite-dimensional), we choose a $G$-$CW$-structure, uniquely 
up to $G$-homotopy equivalence (the equivariant $CW$ 
approximation theorem). The colimit of the homotopy fixed-point sets of the (chains on the) finite-dimensional skeleta then defines a restricted fixed-point set for chains on $S$.
\end{remark}

\subsection{Filtrations.} Floer cohomology involves the addition of a grading variable $q$ of degree $2$, and 
completion with respect to a Novikov parameter $v$ representing the exponentiated area of curves. In the monotone case, 
$v$ is not needed, and the Floer complex is defined over the Laurent polynomial ring $\bC[q^{\pm}]$. 

\section{Summary of the proofs of Theorems 1 and 2}
The ideal proof would present $X$ as a principal $LG$-bundle over $X/\!/G$. While there is no such \emph{geometric} fibration, 
the Floer complex~$\Phi_X$ of $X$,  built from periodic orbits of the Hamiltonian $H:=\frac{1}{2}K\|\mu\|^2$, will provide a 
substitute. Specifically, its $K\to \infty$ limit will land in a neighborhood $N\supset \mu^{-1}(0)$, identified 
symplectomorphically with the $T^*G$-disk bundle over $X/\!/G$ by the \emph{normal form theorem} \cite{gs1}. The 
desired fiber $LG$ will materialize homologically (in an associated  graded complex) as the 
string topology, sor symplectic cohomology, of $T^*G$. (In the case of orbifold quotients, the complex will fiber instead 
over the inertia stack $I(X/\!/G)$ see \cite{chenruan} for its  definition.)

\subsection{Movie: Floer continuation.} \label{movie}
Before proceeding with the outline, we give the heuristics. Let us `turn on' $H =\frac{1}{2}K\|\mu\|^2$ and increase $K$.  
Floer continuation under growing~$K$ flows the periodic Hamiltonian orbits downward along the gradient. In an unstable 
$\|\mu\|^2$-Morse stratum, orbits flow initially to the critical set, accruing (co)homology there. \emph{In the $G$-monotone 
case}, all cohomology then exits again (in a broken flow) along the downward gradient, eventually escaping into the stable 
Morse stratum $X^s$. This exit is enabled by the positive weights on the canoncial bundle of the Hilbert-Mumford destabilizing 
circles. Once in $X^s$, the orbits are drawn into the normal neighborhood $N \supset \mu^{-1}(0)$, and our earlier 
description of $\Phi_X$  takes over. 

This movie is only a guide: continuation need not define continuous maps between orbit spaces. Our argument relies 
instead on a topological estimate for the \emph{monotone index}. For non-abelian $G$, will require the stronger assumption 
of \emph{$G$-monotonicity} of $X$ \eqref{supermonotone}. The movie becomes more relevant away from the $G$-monotone case, 
when some cohomology remains \emph{trapped} at critical loci. 

\subsection{Executive summary: derived additive statement.}\label{additive}
We will define a filtration of $\Phi_X$ compatible with powers of the grading variable $q$. Once in $N$, the associated graded 
complex will fiber over $X/\!/G$, with fiber $SH^*(T^*G)$, a degree-shifted copy of $H_*(LG)$. On the $G$-equivariant complex, 
the induced spectral sequence will 
be shown to collapse at $E_1$. In the free case, this can be seen (with rational coefficients) from Kirwan' surjectivity statement on classical cohomology. On the $LG$-equivariant complex, an isomorphism with the base can be seen integrally, 
by invoking the section of constant orbits $\mu^{-1}(0)/G$.

Locally free actions need an additional check, as we must identify $QH^*_{LG}(X)$ (additively) with the 
cohomology of the inertia stack $I(X/\!/G)$. The twisted sectors of the latter do not come from 
$\mu^{-1}(0)/G$; but neither, for that matter, do the twisted components of $\Phi_X$, consisting of Hamiltonian orbits covered by \emph{fractional} circles in $G$. The replacement for the 
constant section, exhibiting our additive isomorphism, come instead from \emph{$g$-twisted} Floer 
complexes of $X$, with group elements $g$ that occur as stabilizers in $\mu^{-1}(0)$. 
Quasi-isomorphy (now with $\bQ$ coefficients) can be deduced from classical $K$-theoretic Kirwan surjectivity \cite{harl}.


\begin{remark}[Important fine print]
We do \emph{not} filter $\Phi_X$ by $q$-shifts of the classical lattice $H^*_G(X)[q]
\subset H^*_G(X)[q,q^{-1}]$; rather, we will use the \emph {normalized Floer degrees} of orbits in \S\ref{floermix} below.  
\end{remark}

\subsection{Multiplication.} \label{twistedcor}
In the case of free quotients, we see the ring structure from the Lagrangian correspondence 
$\mu^{-1}(0)\subset X\times X/\!/G$ investigated by Wehrheim and Woodward \cite{ww}, and 
equivariantly by Fukaya \emph{et al.}~\cite{f, f2, xiao}. To see its compatibility with our additive isomorphism, we include the correspondence 
equivariantly in the fibration over the group manifold $G$ and deform it by $\frac{1}{2}K\|\mu\|^2$. 
The $G$-equivariant Floer self-homology of the Lagrangian is identified with that of $X/\!/G$. 
As it is also identified with the $LG$-equivariant Floer homology of~$X$, the bimodule it induces for the two sides 
defines the same quasi-isomorphism as the constant section.

In the orbifold case, we relate an inertia stack sector twisted 
by $g\in G$ with an analogous correspondence placed over the 
twisting group elements of $G$: $\mu^{-1}(0)$ is replaced by the moment pre-image of the smallest $g$-fractional orbits.  In this case, however, the correspondence is known to involve a deformation between the two sides.

\subsection{Formality.}
Strictness of the tensor product \eqref{formal} requires a closer look at our $q$-filtration.
The $G$-equivariant complex $\Phi^G_X$ computes $QH^*_G(X)$, and 
higher Tor groups are ruled out by noting  that 
\begin{enumerate}\itemsep0ex
\item $QH^*_G(X)$ is a Cohen-Macaulay module over $H_*^G(G)$,
\item $\mathrm{Spec}\, QH^*_G(X)\cap \mathrm{Spec}\,H^*(BG)$ is a $0$-dimensional subscheme of 
$\mathrm{Spec}\, H_*^G(\Omega G)$.
\end{enumerate}
Claim (i) follows becase $QH^*_G(X)$ is finite and free, under projection to the submanifold $\mathrm{Spec}\,H^*(BG)$. For the
second claim, consider the fibration of $\Phi^G_X$ over the base $I(X/\!/G)$: its filtration spectral sequence collapses at $E_1$ 
to $H^*\left(I(X/\!/G); H_*(\Omega G)\right)$, with the constant section giving the obvious copy of the base therein. This 
$\Omega G$-fibration, associated to the adjoint action of $G$, is compatible with the $H_*^G(\Omega G)$-action. Universally, 
over $BG$, the respective Leray spectral sequence collapses rationally, so that the cohomology over 
$X/\!/G$ can also be split, and the $E_1$ term can be built from finitely many copies of $H_*(\Omega G)$ over the base $I(X/\!/G)$. 
This makes the intersection with $H^*(BG)$ in $H_*^G(\Omega G)$ zero-dimensional, as claimed.

\subsection{Floer ingredients.} \label{floermix}
To spell out the details in execution, we recall some ingrediets of a Hamiltonian Floer complex with Morse-Bott critial loci:

\begin{enumerate}[label = (\alph*)]\itemsep0ex
\item The Floer cohomology grading $\deg$ on an isolated periodic orbit is valued in the $\bZ$-torsor of 
trivalizations of $\det^{\otimes 2}(TX)$ along the parametrizing circle; $q$ shifts trivializations. 

\item On a space of orbits, the ordinary degree of (equivariant) cohomology classes is adjusted by (twice) 
the \emph{Conley-Zehnder shift}, the (floor of the) winding number of the 
Hamiltonian flow on $\det$ (relative to a chosen trivialization). 

\item Framing $TX$ by the action of the covering circle in $G$ normalizes this shift and gives each orbit a 
\emph{normalized Floer degree} $\mathrm{ndeg}$.\footnote{There is 
some ambiguity for fractionally covered orbits; but those fall into finitely many equivalence classes, 
wherein degrees can be integrally reconciled. This ambiguity accounts for the bounded constant C in the formula.}
In the the stable locus, and orbit with velocity vector a dominant $\xi\in \frt$ has
\[
\mathrm{ndeg}(\xi) = -\langle 4\rho | \xi\rangle + C
\] 
This is also the Floer degree in the symplectic cohomology $SH^*(T^*G)$, and the (negative) 
dimension of the Morse cell in $\Omega G$ attached to the closed geodesic $\xi$.

Integrally covered orbits in the unstable locus can be perturbed into the stable locus, when the latter is non-empty; so the 
formula extends to those as well. In the $G$-monotone case, the stable locus cannot be empty. 

\item The \emph{action} $\cA$ of a loop is the sum of the integral of $H$ and a real 
lift of the holonomy of $\cL$. It is valued in the torsor of topological trivializations of $\cL$ along the loop.

\item For a Floer orbit $O_x$, the action is a real lift of the return isomorphism for the Hamiltonian. The covering circle in $G$ can be used to (fractionallly) trivialize $\cL$, leading to the formula
\[
\cA(O_x) = H(x) - \langle\mu(x) | K\mu(x) \rangle 
= -\frac{K}{2} \|\mu\|^2(x)
\] 
having accounted for the velocity of the Hamiltonian flow $dH = \langle K\mu | d\mu\rangle$. 
For a more general Hamiltonian defined by a $G$-invariant function $F(\mu)$,  
\[
\cA = F\left(\mu\right) - \langle dF | \mu \rangle.
\]

\item When the symplectic class agrees with $c_1$, the following  \emph{monotone index} of a Floer orbit is  defined without additional choices, since the two terms are valued in the same torsor:
\[
\mathrm{mix} = \mathrm{Floer\ degree} - 2 \cA. 
\]
\item From (c)--(f) above we get the formulas
\[\begin{split}
\mathrm{mix} &= \mathrm{ndeg} + K\|\mu(x)\|^2 \quad\text{when } 
	H = \frac{1}{2} K\|\mu\|^2,\\
\mathrm{mix} &= \mathrm{ndeg} + 2\left(\langle dF | \mu \rangle- F\left(\mu\right)\right)
	\quad\text{in general.}
\end{split}
\]
\end{enumerate}

\begin{theorem}\label{techtheorem}
\begin{enumerate}\itemsep0ex
\item When $X$ is $G$-monotone, all of $CF^*_G(X)$ eventually reaches the normal neighborhood $N$, 
as $K\to \infty$.
\item For free actions, the section of constant orbits in the $LG$-equivariant Floer complex induces a quasi-isomoprhism with the Floer complex of the base $X/\!/G$.
\item For locally free actions, the twisted sectors of $I(X/\!/G)$ 
define quasi-isomorphic sections of the fractional sectors of $CF^*_{LG}(X)$.
\item Modulo positive $q$-powers with respect to normalized Floer degree (c), the limiting Floer complex agrees with the cohomology complex of $X/\!/G$ with coefficients $SH^*(T^*G) = H_{\dim G -*}(LG)$. 
\item In the presentation (iv), the monodromy $LG$-action (\S\ref{topology}) is the natural $LG$-action on $SH^*$.
\end{enumerate}     
\end{theorem}

\begin{remark}Statements (ii) and (iii) imply the collapse of the $q$-filtration and Leray spectral sequence, 
and the strictness of the tensor product presentation.
\end{remark}

\begin{proof}[Proof of \ref{techtheorem}.i]
The action increases under continuation, so mix decreases. For a torus, $\mathrm{ndeg} =0$, 
and cohomology classes on orbit spaces only increase degree; the first formula in 
(\ref{floermix}.g) then implies that $\mu\to 0$ as $K\to \infty$, so the orbits continue 
into $N$ (and an effective bound can be specified). 

For the non-Abelian case, define $F(\mu): = \frac{K}{2}\min_{g\in G} \|g.\mu-2\rho\|^2$, 
in terms of the squared distance to the co-adjoint orbit of $2\rho$. We use the second formula 
in (\ref{floermix}.g), having assumed that $\mu$ lands in the dominant Weyl chamber, to get 
(leaving out the ambiguity $C$)
\[
\mathrm{mix} = -4K\langle\rho | \mu-2\rho \rangle + 2K\langle \mu-2\rho | \mu \rangle 
	- K\|\mu-2\rho\|^2 = K\|\mu-2\rho\|^2. 
\]
This seemingly forces all orbits to flow to the orbit of $2\rho$. However, we 
left out the smoothing of~$F$ on the singular co-adjoint orbits, which keeps those 
orbits flowing instead, along their walls in the Weyl chamber, to the projections of $2\rho$. Under the $G$-monotone assumption, 
all limiting orbits may be included in the normal neighnorhood $N$ and the argument proceeds as before.
\end{proof}

\begin{proof}[Outline proof of \ref{techtheorem}.ii, iii]
The $g$-twisted forms, $g\in G$, of $\Phi_X$ as $K\to\infty$ are fibered equivariantly 
over the group manifold $G$. Over each $g\in G$, the space of orbits comprises the geodesics 
in $G$ from $p$ to $pg$, as $p$ ranges in $G$, embedded in the fibers $T^*G$ of the normal neighborhood $N$. 

Over all of $G$, the total space of Floer orbits is the space $G\times \mathfrak{g}$ of complete geodesics. Quotienting by $G$ 
leaves the space $I(X/\!/G)\times \mathfrak{g}$. The section at the constant orbits is compatible with the (zero) Floer 
differential: a non-zero differential would violate Kirwan surjectivity.  

Contractibility of $\frg$ secures the desired quasi-isomorphism in the Abelian case. When $\frg$ is non-abelian, this argument is 
incomplete: the normalized degree ndeg jumps over $\frg$, as geodesics acquire conjugate points, so that the Morse flow cannot 
retract the space to the critical locus. The workaround is to observe that contractibility allows us to (homologically) retract 
the critical space $\frg$ to the region of normalized degree zero without changing the homology of the total complex, whereupon the entire complex retracts there.  
\end{proof}


\begin{remark}
A swindle in this beautiful argument was to invoke a well-defined geometric complex for 
$K=\infty$. The work-around, to be explained in Part~II of the paper, performs a $K$-dependent \emph{action
truncation} of $\Phi_X$, with the effect of selecting the part contained in $N$;  
the limiting Floer complex is the $K\to \infty$ colimit of the truncations.    
\end{remark}

\begin{proof}[Proof of \ref{techtheorem}.iv]
The \emph{a priori energy estimate} shows that a Floer differential which is not of topological 
origin increases the action strictly. The constant orbits within $\mu^{-1}(0)$ have the 
maximal action $0$ with respect to the normalization of \S\ref{floermix}.e above, so Floer 
differentials originating there must land in positively $q$-shifted orbits. The same then applies 
to their transforms under shift operators from $1$-parameter subgroups of $G$. When the action of $G$ is free, that settles all the 
orbits. In the locally free case, we apply the same argument to the orbits of maximal action within the fractional equivalence 
class and their shifts. 
\end{proof} 

\begin{proof}[Proof of \ref{techtheorem}.v]
Tautological from the definitions of the $LG$-actions.
\end{proof}

\begin{proof}[Sketch of proof of Theorem~\ref{extra}.i]
On Floer orbits that do \emph{not} approach $\mu^{-1}(0)$, the action necessarily goes to $\infty$ 
with $K$, under continuation. The \emph{a priori} energy estimate shows them to form a subcomplex and 
an ideal for the Floer product. Quotienting out by that ideal leaves a quotient complex,  
to which our earlier analysis carries over. The short exact sequence of complexes leads to a long exact sequence
\[
\dots  \xrightarrow{\ }  tHF^*_{LG} \xrightarrow{\iota_* } QH^*_{LG}(X) \xrightarrow{q_* } A^*(X/\!/G) \xrightarrow{\ } \dots
\]
Its splitting into short exact sequences follows by ruling out equivariant differentials 
originating near $\mu^{-1}(0)$ and leading to trapped orbits. 

On $\mu^{-1}(0)$, the 
(constant) orbits are already present in the small~$K$ Floer complex (the Morse complex for $\|\mu\|^2$), 
so no differentials originate there (for $\bQ$-cohomology). The action of shift operators of $H_*^G(\Omega G)$ settles all orbits. 
A twisted component of the inertia stack $I(X/\!/G)$ is similarly covered 
by fractional orbits, and is matched under the twisted Lagrangian correspondence (Remark~\ref{twistedcor}). The correspondence 
determines classes in $QH^*_{LG}(X)$, showing the lifting of all classes from the twisted sector and (in light of Theorem~\ref{techtheorem}.iii) the vanishing of differentials. 
\end{proof}

The proof of Theorem~\ref{extra}.ii, the agreement of the quotient $A^*$ with the $LG$-equivariant symplectic cohomology of~$X$ 
with supports on $\mu^{-1}(0)$, is postponed to Part~II of the paper. 

\begin{remark}
In the general case, our strictness argument for the tensor product 
\[
QH^*_{LG} \cong QH^*_G\otimes^L_{H_*^G(\Omega G)} H^*(BG) 
\]
only applies to the $G$-equivariant cohomology of the quotient Floer complex modulo trapped classes. 
It is conceivable that some branches of $QH^*_G(X)$ meet the section 
$\mathrm{Spec}\, H^*(BG) \subset \mathrm{Spec}\,H^G_*(\Omega G)$ 
in a positive dimension, leading to Tor groups in the tensor product. Notably, this is the case for trivial $G$-actions.  
\end{remark}

\noindent
Daniel Pomerleano, U Mass.~Boston, 
\texttt{daniel.pomerleano@umb.edu}

\noindent
Constantin Teleman, UC Berkeley, \texttt{teleman@berkeley.edu}

\end{document}